\documentclass{amsart}
\usepackage{amssymb,amsmath,latexsym}
\usepackage{amsthm}
\usepackage{fontenc}
\usepackage{amssymb}

\numberwithin{equation}{section}

\newtheorem{theorem}{Theorem}[section]
\newtheorem{corollary}{Corollary}[theorem]
\newtheorem{lemma}[theorem]{Lemma}
\newtheorem{proposition}[theorem]{Proposition}

\begin{document}
\author{Alexander E. Patkowski}
\title{On summation formulas in probability theory}

\maketitle

\begin{abstract} We offer some summation formulas that appear to have great utility in probability theory. The proofs require some recent results from analysis that have thus far been applied to basic hypergeometric functions.\end{abstract}



\section{Introduction and main theorem} 
In considering discrete random variables, there are a great many useful formulas for calculating probabilistic quantities. For example, a well known fact is if $X$ is discretely distributed, then $\mathop{\mathbb{E}[X]}=\sum_{n\ge0}nP(X=n)=\sum_{n\ge0}P(X>n).$ Another interesting example is if $X,$ and $Y$ are discretely distributed, then
\begin{equation}\sum_{n=0}^{\infty}P(X\le n)P(Y=n)=1-\sum_{n=0}^{\infty}P(X=n+1)P(Y\le n). \end{equation}
A proof of (1.1) may be accomplished as follows. Let
$$ Z:=\{(x,y): x\le y,  x\in X,y\in Y\},$$
then the left side of (1.1) represents $\sum_{n\ge0}P(Z=n)$ if $X$ and $Y$ are independent. Taking the complement of the left side with the concept that $P(Z^{c})=1-P(Z),$ we obtain the right hand side. Interestingly, (1.1) also has a proof using summation by parts, a simple tool from analysis.
\\*
\par The objective of this paper is to apply some recent results in analysis regarding power series to obtain summation formulas for probability theory. In order to prove our main theorem we will require a result due to Andrews and Frietas [1].

\begin{proposition} ([1, Proposition 1.2]) Let $f(z)=\sum_{n=0}^{\infty} \alpha_n z^n$ be analytic for $|z|<1,$ and assume that for some positive integer $N$ and a fixed
complex number $\alpha$ we have that (i) $\sum_{n=0}^{\infty}(n+1)_{N}\left(\alpha_{N+n}-\alpha_{N+n-1}\right)$ converges,  and (ii) $\lim_{n\rightarrow \infty} (n+1)_{N}(\alpha_{N+n}-\alpha)=0.$  Then 
$$\frac{1}{N}\lim_{z\rightarrow1^{-}}\left(\frac{\partial^N}{\partial z^N}(1-z)f(z)\right)=\sum_{n=0}^{\infty}\prod_{j=1}^{N-1}(n+j)\left(\alpha-\alpha_{n+N-1}\right).$$
\end{proposition}
This formula has been fruitful in applications to basic hypergeometric functions. As the present author has shown in [5], it has applications to other areas of analysis. Our main result is the next theorem, and generalizes the expected value.
\begin{theorem} Let N be a positive integer. We have,
$$\begin{aligned}\frac{1}{N}\lim_{z\rightarrow1^{-}}\left(\frac{\partial^N}{\partial z^N}\mathop{\mathbb{E}[z^X]})\right)=&\sum_{n=0}^{\infty}\prod_{j=1}^{N-1}(n+j)P(X>n+N-1)\\ &=\frac{1}{N}\mathop{\mathbb{E}[X(X-1)\cdots(X-N+1)]}. \end{aligned}$$
\end{theorem}
\begin{proof} We apply Proposition 1.1 and set $\alpha_n=P(X\le n),$ for a discrete random variable $X.$ The right side is clear since $(\alpha-\alpha_{n+N-1})=(1-P(X\le n+N-1))=P(X> n+N-1).$ Now by summation by parts (or shifting summation indexes) we know that
\begin{equation} \sum_{n=0}^{\infty}P(X=n)z^n=(1-z)\sum_{n=0}^{\infty}P(X\le n)z^n.\end{equation}
On the other hand, the left side of (1.2) is precisely $\mathop{\mathbb{E}[z^X]},$ the probability generating function for $P(X=n).$ The result now follows
\end{proof}
Next we consider an application of our main theorem by appealing to From's work on discrete random variables [4], where it was found that the factorial moment provides for the most effective upper bound for the survival function $P(X\ge x).$
\begin{corollary} For $x>0,$ we have that
$$P(X\ge x)\le \inf_{0\le N<x}\frac{(N+1)\sum_{n=0}^{\infty}\prod_{j=1}^{N}(n+j)P(X>n+N)}{x(x-1)\cdots(x-N)}.$$
\end{corollary}

\begin{proof}This follows from [4, pg.214, eq.(3)--(4)] together with Theorem 1.1.
\end{proof}
Another interesting identity follows from considering [6, Theorem 5.2]. Let $S_n=\sum_{i\ge1}^{n}X_i,$ where $X_i$ are discrete random variables. As in [6, pg.334, eq.(5.1)--(5.2)], define:
\\* 
\\*
(i) $\eta_n$ to be the number of $i$ such that $S_i>0,$ and $1\le i \le n.$
\\*
(ii) $T_n$ to be the first $i$ such that $\max_{1\le i\le n}S_i$ is achieved, and if $0\ge\max_{1\le i\le n}S_i$ then $T_n=0.$
\\*
Let $m_n=P(S_n>0),$ then if $\sum_{n\ge0}m_n/n$ converges, [6, Theorem 5.2] says that $\eta_n\rightarrow \eta,$ $T_n\rightarrow T$ as $n\rightarrow\infty.$ It follows that Theorem 1.1 together with [6, Theorem 5.2] implies, 
\begin{equation}\sum_{n=0}^{\infty}\prod_{j=1}^{N-1}(n+j)P(\eta>n+N-1)=\sum_{n=0}^{\infty}\prod_{j=1}^{N-1}(n+j)P(T>n+N-1).\end{equation}

\section{The Kolmogorov-Prokhorov formula}

The main object of this section is to create what appears to be a new Kolmogorov-Prokhorov formula based on the ideas presented in [2]. To accomplish this task, we will recall some probabilistic concepts on convergence of sequences of random variables.
\begin{lemma} ([3, pg.129, Definition 6.1.1]) The sequence ${X_n}$ converges in probability to $X$ if
$$P(|X-X_n|>\epsilon)\rightarrow0,$$
for any $\epsilon>0,$ as $n\rightarrow\infty.$
\end{lemma}
 
\begin{lemma} ([3, pg.132, Definition 6.1.3]]) The random variable $X_n$ converges to $X$ in the $r$-order mean if
$$\mathop{\mathbb{E}|X-X_n|^r}\rightarrow0,$$
as $n\rightarrow\infty.$ 
\end{lemma}

\begin{lemma} ([3, pg.138, Theorem 6.1.8]) Suppose $X_n$ converges to $X$ in probability, and let $g(x)$ be continuous with respect the the values of the random variable $X.$ If $g(X_n)$ is uniformly integrable, then
$$\mathop{\mathbb{E}|g(X)-g(X_n)|}\rightarrow0,$$
as $n\rightarrow\infty.$ Furthermore, $\mathop{\mathbb{E}[g(X_n)]}\rightarrow\mathop{\mathbb{E}[g(X)]}$ as $n\rightarrow\infty.$

\end{lemma}

Recall that an integer-valued random variable is a type of discrete random variable. The Kolmogorov-Prokhorov formula [3, pg.77, Theorem 4.4.1] says that if $w$ is an integer-valued  random variable, not dependent on the future (i.e. [3, pg.75, Definition 4.4.1]), and $\sum_{n\ge1}P(w\ge n)\mathop{\mathbb{E}|X_n|}$ converges, then 
\begin{equation}\mathop{\mathbb{E}[S_w]}=\sum_{n=1}^{\infty}P(w\ge n)\mathop{\mathbb{E}[X_n]}. \end{equation}
The formula (2.1) is useful in studying Markov random variables [3, pg.76].

\par From [2, Lemma 3.3] we find a variation on [1, Proposition 1.2]. Suppose that the power series in the variable $z$ of the sequences $\alpha_n,$ $\beta_n,$ and $\alpha_n\beta_n,$ are all analytic for $|z|<1,$ and that $\sum_{n=0}^{\infty}(\alpha-\alpha_{n})$ and $\sum_{n=0}^{\infty}(\beta-\beta_{n}),$ as well as $\sum_{n=0}^{\infty}(\alpha\beta-\alpha_{n}\beta_n)$ converge. Assume $\lim_{n\rightarrow \infty} n(\alpha_{n}-\alpha)=\lim_{n\rightarrow \infty} n(\beta_{n}-\beta)=0,$ and also $\lim_{n\rightarrow \infty} n(\alpha_n\beta_{n}-\alpha\beta)=0.$ Then,
\begin{equation}-\sum_{n=0}^{\infty}\beta_n\left(\alpha-\alpha_n\right)=\alpha\sum_{n=0}^{\infty}\left(\beta-\beta_n\right)-\sum_{n=0}^{\infty}\left(\alpha\beta-\alpha_n\beta_n\right).\end{equation}

\begin{theorem} Suppose $X_n$ converges in mean to $X.$ For $w$ an integer-valued  random variable not dependent on the future,
$$\mathop{\mathbb{E}[S_w]} =\sum_{n=1}^{\infty}\left(\mathop{\mathbb{E}[X]}-\mathop{\mathbb{E}[X_n]}P(w<n)\right)-\sum_{n=1}^{\infty}\left(\mathop{\mathbb{E}[X]}-\mathop{\mathbb{E}[X_n]}\right), $$
provided the series on the right side converge, and the stopping time $w$ is finite.
\end{theorem}

\begin{proof} By [3, pg.132], since $X_n$ converges in mean, it follows that $X_n$ also converges in probability by Chebyshev's inequality. Hence we may apply Lemma 2.3 ([3, pg.138, Theorem 6.1.8]). We choose $g(x)=x,$ and select $\alpha_n=P(w<n+1)$ and $\beta_n=\mathop{\mathbb{E}[X_{n+1}]}$ in (2.2). Note that $\lim_{n\rightarrow\infty}P(w<n+1)=1$ by [3, pg.75, Definition 4.4.1], and the hypothesis that the stopping time $w$ is finite. This gives us the theorem after some rearranging.

\end{proof}

Some remarks on convergence of the series in Theorem 2.4 are worth noting. Since $0\le P(w<n)\le1$ for each $n,$ we know that if $X_n\ge0$ almost everywhere, then
$$|\mathop{\mathbb{E}[X]}-\mathop{\mathbb{E}[X_n]}|\ll |\mathop{\mathbb{E}[X]}-\mathop{\mathbb{E}[X_n]}P(w<n)|.$$
The sequence $X_n$ would then satisfy convergence in mean as well as convergence of the series in question if we imposed the condition 
$$|\mathop{\mathbb{E}[X]}-\mathop{\mathbb{E}[X_n]}P(w<n)|\ll n^{-\epsilon-1},$$
for every $\epsilon>0,$ by a simple comparison test with $\zeta(\epsilon+1)=\sum_{n\ge1}n^{-\epsilon-1}.$

1390 Bumps River Rd. \\*
Centerville, MA
02632 \\*
USA \\*
E-mail: alexpatk@hotmail.com, alexepatkowski@gmail.com

\end{document}